\documentclass[12pt,english]{smfart}

\usepackage{smfthm}
\usepackage{amssymb}

\newcommand{\Qp}{\mathbf{Q}_p}
\newcommand{\Zp}{\mathbf{Z}_p}
\newcommand{\Fp}{\mathbf{F}_p}
\newcommand{\ZZ}{\mathbf{Z}}

\newcommand{\Fpbar}{\overline{\mathbf{F}}_p}
\renewcommand{\geq}{\geqslant}
 
\newcommand{\End}{\operatorname{End}}
\newcommand{\Sym}{\operatorname{Sym}}
\newcommand{\ind}{\operatorname{ind}}

\newcommand{\GL}{\operatorname{GL}}
\newcommand{\Id}{\operatorname{Id}}
\newcommand{\smat}[1]{\left( \begin{smallmatrix} #1 \end{smallmatrix} \right)}
\newcommand{\G}{\operatorname{GL}_2(\Qp)}
\newcommand{\B}{\operatorname{B}_2(\Qp)}
\newcommand{\K}{\operatorname{GL}_2(\Zp)}
\newcommand{\Z}{\operatorname{Z}}
\newcommand{\Sp}{\operatorname{Sp}}
\renewcommand{\ss}{\operatorname{ss}}

\author{Laurent Berger}
\address{UMPA, ENS de Lyon \\
UMR 5669 du CNRS \\
Universit\'e de Lyon}
\email{laurent.berger@ens-lyon.fr}
\urladdr{perso.ens-lyon.fr/laurent.berger/}

\date{July 2011}

\title[Central characters for representations of $\G$]
{Central characters for smooth irreducible modular representations of $\operatorname{GL}_2(\Qp)$}



\begin{document}

\begin{abstract}
We prove that every smooth irreducible $\overline{\mathbf{F}}_p$-linear representation of $\operatorname{GL}_2(\Qp)$ admits a central character.
\end{abstract}

\maketitle

\setlength{\baselineskip}{18pt}

\section*{Introduction}

Let $\Pi$ be a representation of $\G$. We say that $\Pi$ is smooth, if the stabilizer of any $v \in \Pi$ is an open subgroup of $\G$. We say that $\Pi$ admits a central character, if every $z \in \Z(\G)$ acts on $\Pi$ by a scalar. The smooth irreducible representations of $\G$ over an algebraically closed field of characteristic $p$, admitting a central character, have been studied by Barthel-Livn\'e in \cite{BL2,BL} and by Breuil in \cite{BR1}. The purpose of this note is to prove the following theorem.

\begin{enonce*}{Theorem A}
If $\Pi$ is a smooth irreducible $\Fpbar$-linear representation of $\G$, then $\Pi$ admits a central character.
\end{enonce*}

The idea of the proof of theorem A is as follows. If $\Pi$ does not admit a central character, and if $f = \smat{p&0\\0&p}$, then for any nonzero polynomial $Q(X) \in \Fpbar[X]$, the map $Q(f) : \Pi \to \Pi$ is bijective, so that $\Pi$ has the structure of a $\Fpbar(X)$-vector space. The representation $\Pi$ is therefore a smooth irreducible $\Fpbar(X)$-linear representation of $\G$, which now admits a central character, since $f$ acts by multiplication by $X$. It remains to apply Barthel-Livn\'e and Breuil's classification, which gives the structure of the components of $\Pi$ after extending scalars to a finite extension $K$ of $\Fpbar(X)$. A corollary of this classification is that these components are all ``defined'' over a subring $R$ of $K$, where $R$ is a finitely generated $\Fpbar$-algebra. This can be used to show that $\Pi$ is not of finite length, a contradiction.

\noindent\textbf{Acknowledgments}. I am grateful to C.\ Breuil, G.\ Chenevier, P.\ Colmez, G.\ Henniart, M.\ Schein and M.-F.\ Vign\'eras for helpful comments.

\section{Barthel-Livn\'e and Breuil's classification}
\label{bbl}

Let $E$ be a field of characteristic $p$. In this section, we recall the explicit classification of smooth irreducible $E$-linear representations of $\G$, admitting a central character. 

We denote the center of $\G$ by $\Z$. If $r \geq 0$, then $\Sym^r E^2$ is a representation of $\GL_2(\Fp)$ which gives rise, by inflation, to a representation of $\K$. We extend it to $\K\Z$ by letting $\smat{p&0\\0&p}$ act trivially. Consider the representation 
\[Ê\ind_{\K\Z}^{\G} \Sym^r E^2. \] 
The Hecke algebra
\[ \End_{E[\G]} \left( \ind_{\K\Z}^{\G} \Sym^r E^2 \right) \] 
is isomorphic to $E[T]$ where $T$ is a Hecke operator, which corresponds to the double class $\K\Z \cdot \smat{p & 0 \\ 0 & 1} \cdot \K$. If $\chi : \Qp^\times \to E^\times$ is a smooth character, and if $\lambda \in E$, then let
\[ \pi(r,\lambda,\chi) = \frac{ \ind_{\K\Z}^{\G} \Sym^r E^2}{T-\lambda} \otimes (\chi \circ \det). \]
This is a smooth representation of $\G$, with central character $\omega^r \chi^2$ (where $\omega : \Zp^\times \to \Fp^\times$ is the ``reduction mod $p$'' map). Let $\mu_\lambda : \Qp^\times \to E^\times$ be given by $\mu_\lambda \mid_{\Zp^\times} = 1$, and $\mu_\lambda(p)=\lambda$. If $\lambda = \pm 1$, then we have two exact sequences:
\begin{gather*} 
0 \to \Sp \otimes (\chi\mu_\lambda \circ \det) \to \pi(0,\lambda,\chi) \to \chi\mu_\lambda \circ \det \to 0, \\
0 \to \chi\mu_\lambda \circ \det \to \pi(p-1,\lambda,\chi) \to \Sp \otimes (\chi\mu_\lambda \circ \det) \to 0,
\end{gather*}
where the representation $\Sp$ is the ``special'' representation.

\begin{theo}\label{bbl2}
If $E$ is algebraically closed, then the smooth irreducible $E$-linear representations of $\G$, admitting a central character, are as follows:
\begin{enumerate}
\item $\chi \circ \det$;
\item $\Sp \otimes (\chi \circ \det)$;
\item $\pi(r,\lambda,\chi)$, where $r \in \{0,\hdots,p-1\}$ and $(r,\lambda) \notin \{(0,\pm 1), (p-1, \pm 1)\}$.
\end{enumerate}
\end{theo}

This theorem is proved in \cite{BL} and \cite{BL2}, which treat the case $\lambda \neq 0$, and in \cite{BR1}, which treats the case $\lambda = 0$. 

We now explain what happens if $E$ is not algebraically closed.

\begin{prop}\label{sir}
If $\Pi$ is a smooth irreducible $E$-linear representation of $\G$, admitting a central character, then there exists a finite extension $K/E$ such that $(\Pi \otimes_E K)^{\ss}$ is a direct sum of $K$-linear representations of the type described in theorem \ref{bbl2}.
\end{prop}

\begin{proof}
Barthel and Livn\'e's methods show (as is observed in \S 5.3 of \cite{P10}) that $\Pi$ is a quotient of 
\[ \Sigma = \frac{ \ind_{\K\Z}^{\G} \Sym^r E^2}{P(T)} \otimes (\chi \circ \det), \]
for some integer $r \in \{0,\hdots,p-1\}$, character $\chi : \Qp^\times \to E^\times$, and polynomial $P(Y) \in E[Y]$. Let $K$ be a splitting field of $P(Y)$, write $P(Y)=(Y-\lambda_1) \cdots (Y-\lambda_d)$, and let $P_i(Y) = (Y-\lambda_1) \cdots (Y-\lambda_i)$ for $i=0,\hdots,d$. The representations $P_{i-1}(T) \Sigma / P_i(T) \Sigma$ are then isomorphic to $\pi(r,\lambda_i,\chi)$, for $i=1,\hdots,d$.
\end{proof}

We finish this section by recalling that if $\lambda \neq 0$, then the representations $\pi(r,\lambda,\chi)$ are parabolic inductions (when $\lambda=0$, they are called supersingular). Let $\chi_1$ and $\chi_2 : \Qp^\times \to E^\times$ be two smooth characters, and consider the parabolic induction $\ind_{\B}^{\G} (\chi_1 \otimes \chi_2)$. The following result is proved in \cite{BL2} and \cite{BL}.

\begin{theo}\label{pic}
If $\lambda \in K \setminus \{0;\pm 1\}$, and if $r \in \{0,\hdots,p-1\}$, then $\pi(r,\lambda,\chi)$ is isomorphic to $\ind_{\B}^{\G}(\chi\mu_{1/\lambda}, \chi\omega^r \mu_\lambda)$.
\end{theo}

\section{Proof of the theorem}
\label{pta}

We now give the proof of theorem A. Let $\Pi$ be a smooth irreducible $\Fpbar$-linear representation of $\G$. We have $\Pi^{(1+p\Zp) \cdot \Id} \neq 0$ (since a $p$-group acting on a $\Fp$-vector space always has nontrivial fixed points), so that if $\Pi$ is irreducible, then $(1+p\Zp) \cdot \Id$ acts trivially on $\Pi$. If $g \in \Zp^\times \cdot \Id$, then $g^{p-1}=\Id$ on $\Pi$, so that $\Pi=\oplus_{\omega \in \Fp} \Pi^{g=\omega \cdot \Id}$. Since $\Pi$ is irreducible, this implies that the elements of $\Zp^\times \cdot \Id$ act by scalars. 

If $f = \smat{p&0\\0&p}$, then for any nonzero polynomial $Q(X) \in \Fpbar[X]$, the kernel and image of the map $Q(f) : \Pi \to \Pi$ are subrepresentations of $\Pi$. If $Q(f)=0$ on a nontrivial subspace of $\Pi$, then $f$ admits an eigenvector for an eigenvalue $\lambda \in \Fpbar$. This implies that $\Pi=\Pi^{f = \lambda \cdot \Id}$, so that $\Pi$ does admit a central character. If this is not the case, then $Q(f)$ is bijective for every nonzero polynomial $Q(X) \in \Fpbar[X]$, so that $\Pi$ has the structure of a $\Fpbar(X)$-vector space, and is a $\Fpbar(X)$-linear smooth irreducible representation of $\G$, admitting a central character. 

Let $E = \Fpbar(X)$. Proposition \ref{sir} gives us a finite extension $K$ of $E$, such that $(\Pi \otimes_E K)^{\ss}$ is a direct sum of $K$-linear representations of the type described in theorem \ref{bbl2}. The $\Fpbar$-linear representation underlying $(\Pi \otimes_E K)^{\ss}$ is isomorphic to $\Pi^{[K:E]}$, and hence of length $[K:E]$. We now prove that none of the $K$-linear representations of the type described in theorem \ref{bbl2} are of finite length, when viewed as $\Fpbar$-linear representations.

Let $\Sigma$ be one such representation, and let $\lambda \in K$ be the corresponding Hecke eigenvalue. 

\begin{prop}\label{defr}
There exists a subring $R$ of $K$, which is a finitely generated $\Fpbar$-algebra, such that $\Sigma = \Sigma_R \otimes_R K$, where $\Sigma_R$ is an $R$-linear representation of $\G$.
\end{prop}

\begin{proof}
If $\lambda \in \Fpbar$, then theorem \ref{bbl2}Ê shows that 
\[ \Sigma = \frac{ \ind_{\K\Z}^{\G} \Sym^r \Fpbar^2}{T-\lambda} \otimes_{\Fpbar} K(\chi \circ \det), \text{ or } \Sp \otimes_{\Fpbar} K(\chi \circ \det), \text{ or } K(\chi \circ \det). \] 
We can then take $R=\Fpbar[\chi(p)^{\pm 1}]$, and $\Sigma_R = (\ind_{\K\Z}^{\G} \Sym^r \Fpbar^2 / (T-\lambda)) \otimes_{\Fpbar} R(\chi \circ \det)$, or $\Sp \otimes_{\Fpbar} R(\chi \circ \det)$, or $R(\chi \circ \det)$, respectively.

If $\lambda \notin \Fpbar$, then by theorem \ref{pic}, we have
\[ \Sigma = \ind_{\B}^{\G}(\chi\mu_{1/\lambda}, \chi\omega^r \mu_\lambda). \]
We can take $R=\Fpbar[\lambda^{\pm 1}, \chi(p)^{\pm 1}]$, and let $\Sigma_R$ be the set of functions $f \in \Sigma$ with values in $R$.
\end{proof}

Let $\beta \in \Fpbar$ be such that $(X-\beta) \notin R^\times$, so that $(X-\beta)^j \Sigma_R \neq (X-\beta)^{j+1} \Sigma_R$ for all $j \in \ZZ$. The representation $\Sigma$ contains $\cup_{ j \in \ZZ} (X-\beta)^j \Sigma_R$, so that the underlying $\Fpbar$-linear representation is not of finite length, which is a contradiction. This finishes the proof of theorem A.

\providecommand{\bysame}{\leavevmode ---\ }
\providecommand{\og}{``}
\providecommand{\fg}{''}

\end{document}